\documentclass[a4paper,10pt]{article}
\usepackage[utf8]{inputenc}
\usepackage[T1]{fontenc}
\usepackage[english]{babel}
\usepackage{amsmath}
\usepackage{amsthm}
\usepackage{amsfonts}
\usepackage{amssymb}
\usepackage{listings}
\usepackage{graphicx}
\usepackage{mathrsfs}
\usepackage{hyperref}
\usepackage{algpseudocode}
\usepackage{algorithmicx}
\usepackage{algorithm}
\usepackage{subfig}
\usepackage{stmaryrd}
\usepackage{bm}
\usepackage{tikz}
\usepackage{braket}
\usepackage{wasysym}
\usepackage{wrapfig}
\usepackage[framemethod=TikZ]{mdframed}
\usepackage{xcolor}
\usepackage{pstricks}
\usepackage{faktor}
\usepackage{ulem}

\usepackage{enumitem}
\usepackage[top=2.00cm,bottom=3.00cm,left=2.00cm,right=2.00cm]{geometry} 
\usepackage{todonotes}

\newcommand{\mc}{\mathscr}

\newcommand{\mf}{\mathfrak}
\newcommand{\f}{\mathbb}

\newcommand{\cu}{\subseteq}

\newcommand{\serie}[1]{\{#1_{n}\}_n}
\newcommand{\GLT}{\sim_{GLT}}
\newcommand{\dacs}[2]{d_{acs}\left(#1,#2\right)}
\newcommand{\acs}{\xrightarrow{a.c.s.}}
\newcommand{\B}{\{B_{n,m}\}_{n,m}}
\newcommand{\ve}{\varepsilon}

\newcommand{\ea}{\equiv_{acs}}

\DeclareMathOperator{\rk}{rk}

\theoremstyle{definition}
\newtheorem{definizione}{Definition}[section]
\theoremstyle{plain}
\newtheorem{theorem}[definizione]{Theorem}

\newtheorem{lemma}[definizione]{Lemma}
\newtheorem{corollary}[definizione]{Corollary}
\newtheorem{conjecture}[]{Conjecture}

\title{Conjectures on Perturbations of Hermitian Sequences}
\author{Giovanni Barbarino}

\begin{document}

\maketitle

\section{Introduction}

A \textit{Matrix-Sequence} $ \serie A $ is an ordered collection of complex matrices such that $ A_n\in \mathbb C^{n\times n} $. 
We will denote by $\mathscr E$ the space of all matrix-sequences, 
\[
\mathscr  E := \{\serie{A} : A_n\in\mathbb C^{n\times n} \}.
\]
When dealing with linear differential equations and their discretization, several matrix sequences appear and often they are associated to some spectral \textit{Spectral Symbol}, that is a measurable function describing the asymptotic distribution of the eigenvalues of the sequence  in the Weyl sense \cite{BS,GLT-book,Tilliloc}. We recall that a spectral symbol associated with a sequence $\serie A$  is a measurable function $f:D\cu \mathbb R^q\to \mathbb C$, $q\ge 1$, satisfying 
\[
\lim_{n\to\infty} \frac{1}{n} \sum_{i=1}^{n} F(\lambda_i(A_n)) = \frac{1}{l(D)}\int_D F(f(x)) dx
\]
for every continuous function $F:\mathbb C\to \mathbb C$ with compact support, where $D$ is a measurable set with finite Lebesgue measure $l(D)>0$ and $\lambda_i(A_n)$ are the eigenvalues of $A_n$. In this case we write 
\[ \serie A\sim_\lambda f. \]

 The computation of such symbol is not trivial, so one can, for example, use the results regarding GLT sequences when dealing with Hermitian matrices, or analyse the singular values in the general case. Few is known about the eigenvalues of non-normal sequences, except in the case where the matrices are small perturbation of Hermitian matrices. 

\begin{theorem}\cite[Theorem~3.4]{golinskii}\label{gol}
	Let $\serie X$, $\serie Y$ be matrix-sequences and set $A_n=X_n+Y_n$. Assume
	that the following conditions are met.
	\begin{itemize}
		\item Every $X_n$ is an Hermitian matrix and $\serie X\sim_\lambda f$.
		\item $\|Y_n\|_{1} = o(n)$.
		\item $\|X_n\|,\|Y_n\|\le C$ for all $n$, where $C$ is a constant independent from $n$.
	\end{itemize}
	Then $\serie A\sim_\lambda f$. 
\end{theorem}

The  applications of the previous result are countless, but further experiments showed that the conditions on the spectral norms seemed technical and irrelevant.  Successive studies led in fact to more powerful results.

\begin{theorem} \cite{Per1}\label{th:pert}
	Let $\serie X$, $\serie Y$ be matrix-sequences and set $A_n=X_n+Y_n$. Assume
	that the following conditions are met.
	\begin{itemize}
		\item Every $X_n$ is an Hermitian matrix and $\serie X\sim_\lambda f$.
		\item $\|Y_n\|_{1} = o(\sqrt{n})$.
	\end{itemize}
	Then $\serie A\sim_\lambda f$. 
\end{theorem}

\begin{lemma} \cite{Per1}\label{th:pert2}
	Let $\serie X$, $\serie Y$ be matrix-sequences and set $A_n=X_n+Y_n$. Assume
	that the following conditions are met.
	\begin{itemize}
		\item Every $X_n$ is an Hermitian matrix and $\serie X\sim_\lambda f$.
		\item $\|Y_n\|_{1} = o(n)$.
		\item $\|Y_n\|\le C$ for all $n$, where $C$ is a constant independent from $n$.
	\end{itemize}
	Then $\serie A\sim_\lambda f$. 
\end{lemma}

Lemma \ref{th:pert2} is a direct generalization of Theorem \ref{gol}, but Theorem \ref{th:pert} holds under different hypothesis.  
The experiments executed in the same document suggested that a more powerful result may hold, namely

\begin{conjecture}\label{conj}
	Let $X_n$ be a Hermitian matrix of size $n$, with $\serie X\sim_\lambda f$. If $\|Y_n\|_1 = o( n)$ then 
	\[ \{X_n+Y_n\}_n \sim_\lambda f. \]
\end{conjecture}
Notice that the questions we want to discuss do not depend on the domain of the symbols, so from now on all the symbols are measurable function on $[0,1]$. Here we show some of the works done on the problem and similar results found on normal sequences.

\section{Prerequisites}

Throughout the paper, we use $\Re$ and $\Im$ to denote real and imaginary part of numbers, functions, matrices and sequences, where
\begin{align*}
k:\f C\to \f C \implies \Re(k) = \frac{k+k^*}{2},\Im(k) = \frac{k-k^*}{2\text i},\qquad &
M\in \f C^{n\times n} \implies 
\Re(M) = \frac{M+M^H}{2},\Im(M) = \frac{M-M^H}{2\text i},\\
\serie A\in \mc E \implies \Re(\serie A) = &\{ \Re(A_n)\}_n,\quad\Im(\serie A) =
\{ \Im(A_n) \}_n.
\end{align*}
We will use the formalization of GLT sequences given in \cite{Barbarino2017}. Fist, we need to survey the theory on spectral symbols, singular value symbols and approximating classes of sequences.
\subsection{Symbols}

A \textit{Singular Value Symbol} associated with a sequence $\serie A$ is a measurable functions $k:D\cu \f R^n\to \f C$, where $D$ is measurable set with finite non-zero Lebesgue measure, satisfying 
\[
\lim_{n\to\infty} \frac{1}{n} \sum_{i=1}^{n} F(\sigma_i(A_n)) = \frac{1}{l(D)}\int_D F(|k(x)|) dx
\]
for every continuous function $F:\f R\to \f C$ with compact support.
Here $l(D)$ is the Lebesgue measure of $D$, and 
\[\sigma_1(A_n)\ge \sigma_2(A_n)\ge\dots\ge \sigma_n(A_n)\]
are the singular values in non-increasing order. In this case, we will say that $\serie A$ has singular value symbol $k$ and we will write
\[\serie A\sim_\sigma k.\]
A \textit{Spectral Symbol} is a measurable function in the Weyl sense describing the asymptotic distribution of the eigenvalues of $\serie A$   \cite{BS,GLT-book,Tilliloc}. It is a function $f:D\cu \mathbb R^q\to \mathbb C$, $q\ge 1$, satisfying 
\[
\lim_{n\to\infty} \frac{1}{n} \sum_{i=1}^{n} F(\lambda_i(A_n)) = \frac{1}{l(D)}\int_D F(f(x)) dx
\]
for every continuous function $F:\mathbb C\to \mathbb C$ with compact support, where $D$ is a measurable set with finite Lebesgue measure $l(D)>0$ and $\lambda_i(A_n)$ are the eigenvalues of $A_n$. In this case we write 
\[ \serie A\sim_\lambda f. \]
The functions $k,f$ in general are not uniquely determined, since the definition are distributional, so two functions with the same distribution are always simultaneously symbols of the same sequence.
In general, we say that a function $k$ is a rearranged version of $h$ if they have the same domain $D$ of finite non-zero Lebesgue measure, and the same distribution, meaning 
\begin{equation}\label{ergodic}
\frac{1}{l(D)}\int_D F(h(x)) dx =  \frac{1}{l(D)}\int_D F(k(x)) dx
\end{equation}
for every continuous function $F:\f C\to \f C$ with compact support. It is clear that given a sequence $\serie A\sim_\lambda h$, it holds 
\[
\serie A\sim_\lambda k \iff k \text{ rearranged version of }h
\]
and if $\serie B\sim_\sigma h$, then
\[
\serie B\sim_\sigma k \iff |k| \text{ rearranged version of }|h|.
\]\\

\noindent Three famous classes of sequences that admit a symbol are the following.
\begin{itemize}
	\item Given a function $f$ in $L^1([-\pi,\pi])$, its associated Toeplitz sequence is $\{T_n(f)\}_n$, where
	\[
	T_n( f ) = [ f_{i-j} ]^n_{i, j=1}, \qquad f_k = \frac{1}{2\pi} \int_{-\pi}^{\pi} f(\theta) e^{-\text i k\theta} d\theta.
	\]
	We know that $\{T_n(f)\}_n\sim_\sigma f$\cite{szego}, and if $f$ is real-valued, then $T_n(f)$ are Hermitian matrices and  $\{T_n(f)\}_n\sim_\lambda f$\cite{vassalos}.
	\item Given any a.e.\ continuous function $a:[0,1]\to\mathbb C$, its associated diagonal sampling sequence is $\{D_n(a)\}_n$, where
	\[
	D_n(a) = \mathop{\rm diag}_{i=1,\ldots,n}a\Bigl(\frac in\Bigr).
	\]
	We get $\{D_n(a)\}_n\sim_{\sigma,\lambda}a(x)$ since $D_n(a)$ are normal matrices.
	\item  A zero-distributed sequence is a matrix-sequence such that $\{Z_n\}_n\sim_\sigma0$, i.e.,
	\[ \lim_{n\to\infty}\frac1n\sum_{i=1}^nF(\sigma_i(Z_n))=F(0) \]
	for every continuous function $F:\mathbb R\to\mathbb C$ with compact support. If $Z_n$ are normal matrices, then $\{Z_n\}_n\sim_\lambda0$ holds too.
\end{itemize}

\subsection{Approximating Classes of Sequences}

Let $\mc C_D$ be the space of matrix sequences that admit a spectral symbol on a fixed domain $D$. It has been shown to be closed with respect to a notion of convergence called the Approximating Classes of Sequences (acs) convergence. This notion and this result are due to Serra-Capizzano, but were
actually inspired by Tilli’s pioneering paper on LT sequences \cite{Tilliloc}. Given a sequence of matrix sequences $\B$, it is said to be acs convergent to $\serie A$ if there exists a sequence $\{N_{n,m}\}_{n,m}$ of "small norm" matrices and a sequence  $\{R_{n,m}\}_{n,m}$ of "small rank" matrices such that for every $m$ there exists $n_m$ with
\[
A_n = B_{n,m}  + N_{n,m} + R_{n,m}, \qquad \|N_{n,m}\|\le \omega(m), \qquad \rk(R_{n,m})\le nc(m)
\]
for every $n>n_m$, and
\[
\omega(m)\xrightarrow{m\to \infty} 0,\qquad c(m)\xrightarrow{m\to \infty} 0.
\]
In this case, we will use the notation $\B\acs \serie A$. The result of closeness can be expressed as
\begin{lemma}
	If $\B \sim_\sigma k_m$ for every $m$, $k_m\to k$ in measure, and $\B\acs \serie A$, then $\serie A\sim_\sigma k$.
\end{lemma}

\begin{lemma}
	If $\B \sim_\lambda k_m$ for every $m$, $k_m\to k$ in measure, and $\B\acs \serie A$, where all the matrices $B_{n,m}$ and $A_n$ are Hermitian,  then $\serie A\sim_\lambda k$.
\end{lemma}

These results are central in the theory since they let us compute the symbols of acs limits, and it is useful when we can find simple sequences that converge to the wanted $\serie A$.

Given a matrix $A\in\f C^{n\times n}$, we can define the function
\[
p(A):= \min_{i=1,\dots,n+1}\left\{ \frac{i-1}{n} + \sigma_i(A) \right\}
\]
where, by convention, $\sigma_{n+1}(A)=0$.
Given a sequence $\serie A\in\mc E$, we can denote
\[
\rho\left(\serie A\right):= \limsup_{n\to \infty} p(A_n).
\]
This allows us to introduce a pseudometric $d_{acs}$ on $\mc E$
\[
\dacs{\serie{A}}{\serie{B}} = \rho\left(\{A_n-B_n\}_n\right).
\]
It has been proved (\cite{joint-Albrecht},\cite{Garoni}) that this pseudodistance induces the acs convergence already introduced. 
Moreover, this pseudodistance is complete over $\mc E$, and consequentially it is complete over any closed subspace (see \cite{Barbarino2017} and \cite{BG-ELA2017} for proof and further details). The completeness is an important property, and it is a corollary of the following results (\cite{TR},\cite{barb}).
\begin{lemma}\label{mn}
	Let $\B$ be a sequence of matrix sequences that is a Cauchy sequence with respect to the pseudometric $d_{acs}$. There exists a crescent map $m:\f N\to \f N$ with $\lim_{n\to\infty} m(n) = \infty$ such that for every crescent map $m':\f N\to \f N$ that respects
	\begin{itemize}
		\item $m'(n)\le m(n) \quad \forall n $
		\item $\lim_{n\to\infty} m'(n) = \infty$
	\end{itemize}
	we get
	\[
	\B \acs \{ B_{n,m'(n)} \}_n.
	\]
\end{lemma}

\begin{lemma}\label{compl}
	Let $d_n$ be pseudometrics on the space of matrices $\mathbb C^{n\times n}$ bounded by the same constant $L>0$ for every $n$. Then the function
	\[
	d(\serie A,\serie B):= \limsup_{n\to \infty} d_n(A_n,B_n)
	\]
	is a complete pseudometric on the space of matrix sequences.
\end{lemma}

It is possible to give a characterization of the zero-distributed sequences as sum of "small norm" and "small rank" sequences.
The following result sums up the important properties of the acs convergence.

\begin{lemma}\label{zero_distributed}
	Given $\serie A\in\mc E$, $\serie B\in \mc E$ and $\{B_{n,m}\}_{n,m}\in \mc E$ for every $m$, the following results hold.
	\begin{enumerate}
		\item The pseudodistance $d_{acs}$ is complete over every  closed subspace of  $\mc E$,
		\item \[
		\dacs{\serie A}{\B} \xrightarrow{m\to \infty} 0 \iff  \B\acs \serie A,
		\]
		\item \[
		\{A_n-B_n\}_n \sim_\sigma 0 \iff \dacs{\serie A}{\serie B} = 0 \iff A_n -B_n = R_n + N_n\,\, \forall\, n
		\]
		where $\rk(R_n) = o(n)$ and $\|N_n\| = o(1)$.
	\end{enumerate}
	
\end{lemma}

\begin{lemma}[Theorem 3.3,\cite{GLT-book}]\label{zero:norm_1}
	If $\serie A\in\mc E$ and $\|A_n\|_1=o(n)$, then $\serie A$ is zero distributed.
\end{lemma}

The last point of Lemma \ref{zero_distributed} shows that an equivalence relation naturally arises from the definition of acs distance. In fact $\serie A$ and $\serie B$ are said to be \textit{acs equivalent} if their difference is a zero-distributed sequence. In this case, we will write  $\serie A\equiv_{acs}\serie B$.

The set of zero-distributed sequence $\mc Z$ is a subgroup (actually a non-unital subalgebra) of the ring $\mc E$ and $d_{acs}$ is  a complete distance on   the quotient $\mc E/\mc Z$. These properties are fully exploited and developed in the theory of GLT sequences, that we are going to summarize in the next section.
\subsection{GLT Algebra}

Let us denote by $\mf C_D$ the set of couples $(\serie A,k)\in \mc E\times \mc M_D$ such that $\serie A\sim_\sigma k$ and where $D=[0,1]\times[-\pi,\pi]$. First of all we can see that it is well defined, because from the definition, if $k,k'$ are two measurable functions that coincide almost everywhere, then 
\[
\serie A\sim_\sigma k\iff \serie A\sim_\sigma k',
\] 

The set of couples of GLT sequences and symbols $\mf G $ is a subset of $\mf C_D$, so when we say that $\serie A$ is a GLT sequence with symbol $k$ and we write $\serie A\GLT k$, it means that $(\serie A,k)\in \mf G$ and in particular, it means that $\serie A\sim_\sigma k$. The set of GLT sequences is denoted with $\mc G$.

The GLT set is built so that for every $\serie A\in \mc E$ there exists at most one function $k$ such that $\serie A\GLT k$, but not every sequence in $\mc E$ is a GLT sequence. To understand what is the GLT space, we have to start introducing its fundamental bricks, that are the already mentioned Toeplitz, diagonal and zero-distributed sequences 
\begin{itemize}
	\item Given a function $f$ in $L^1([-\pi,\pi])$, the GLT symbol of its associated Toeplitz sequence is $f(\theta)$ itself
	\[
	\{T_n(f)\}_n\sim_{GLT} f(\theta).
	\]
	\item Given any a.e.\ continuous function $a:[0,1]\to\mathbb C$, the GLT symbol of its associated diagonal sampling sequence is $a(x)$ 
	\[
	\{D_n(a)\}_n\GLT a(x).
	\]
	\item  Any zero-distributed sequence $\serie Z$ has $0$ as GLT symbol.
	\[
	\serie Z \GLT 0.
	\]
\end{itemize}
Notice that the GLT symbols are measurable functions $k(x,\theta)$ on the domain $D$, where $x\in [0,1]$, $\theta\in [-\pi,\pi]$. Using these ingredients, we can build the GLT space through the algebra composition rules, and the acs convergence. The uniqueness of the GLT symbol let us define a map
\[
S : \mc G\to \mc M_D 
\]
that associates to each sequence its GLT symbol
\[
S(\serie A) = k \iff \serie A\GLT k\iff (\serie A,k)\in \mf G.
\]
We report here the main properties of $\mf G$ and $S$, that can be found in \cite{GLT-book} and \cite{Barbarino2017}, and that let us generate the whole space.
\begin{enumerate}
	\item $\mf G$ is a $\f C$-algebra, meaning that given $(\serie A,k)$,$(\serie B,h)\in \mf G$ and $\lambda \in\f C$, then
	\begin{itemize}
		\item $(\{ A_n+B_n \}_n,k+h)\in \mf G$,
		\item $(\{ A_nB_n \}_n,kh)\in \mf G$,
		\item $(\{ \lambda A_n \}_n,\lambda k)\in \mf G$.
	\end{itemize}
	\item \label{close}
	$\mf G$ is closed in $\mc E\times \mc M_D$: given $\{(\B,k_m)\}_m\cu \mf G$ such that 
	\[\B\acs\serie A, \qquad k_m\to k \text{ in measure},\]
	the couple $(\serie A,k)$ belongs to $\mf G$.
	\item \label{zero}If we denote the sets of zero distributed sequences as
	\[
	\mf Z = \{ (\serie Z,0)| \serie Z\sim_\sigma 0 \},\qquad \mc Z = \{ \serie Z | \serie Z\sim_\sigma 0 \},
	\]
	then $\mf Z$ is an ideal of $\mf G$.
	\item\label{ker} $S$ is a surjective homomorphism of $\f C$-algebras and $\mc Z$ coincides with its kernel. Moreover $S$ respects the metrics of the spaces, meaning that
	\[
	\dacs{\serie A}{\serie B} = d_m(S(\serie A),S(\serie B))
	\]
	where the distance $d_m$ on $\mc M_D$ induces the convergence in measure.
\end{enumerate}
Notice that $S$ links the distance $d_{acs}$ on $\mc G$ and the distance $d_m$ on $\mc M_D$. This property actually holds for every group $\mf A\cu \mf C_D$, and let us identify Cauchy sequences in $\mc E$ from Cauchy sequences on $\mc M_D$ and vice versa. In particular, given $\{(\B,k_m)\}_m\cu \mf A$ we have
\[
\B\acs \serie A \iff k_m\to k \text{ in measure}
\]
so we say that $\mf A$ is closed whenever the set of its sequences is closed in $\mc E$.

Eventually, we also report that the symbols of Hermitian GLT sequences are also spectral symbols.

\begin{lemma}\label{GLT:hermitian_spectral}
	If $\serie A\GLT k$ is an Hermitian sequence, then $\serie A\sim_\lambda k$. 
\end{lemma}

The main result of \cite{barb}, that improves the one in \cite{TR}, deals with diagonal sequences, and tells us that all diagonal sequences with spectral symbol are just permuted version of GLT diagonal sequences.
\begin{theorem}\label{LAMBDAtoGLT}
	Given $\serie D$ a sequence of diagonal matrices such that $\serie{D}\sim_\lambda f(x)$, with $f:[0,1]\to\f C$,
	 there exist permutation matrices $P_n$ such that
	$$\{P_nD_nP_n^T\}_n\GLT f(x)\otimes 1.$$
\end{theorem}

%
%
%
%

\section{The Problem}

We already stated the conjecture, but the opposite problem is also fairly interesting.

\begin{conjecture}\label{inv}
Given $\serie X\sim_\lambda f$,  and $\|\Im(X_n)\|_1=o(n)$. Is it true that $\{ \Re(X_n) \}_n \sim_\lambda f$?
\end{conjecture}
We will tackle the problems from different sides, using tools from measure theory, linear algebra and metric spaces.
First we see some counterexample to similar questions, in order to put bounds on our research and to not go astray.

\subsection{Some Counterexamples}
Notice that if $Y_n$ are zero-distributed, but without the hypothesis on the norm $\|Y_n\|_1=o(n)$, there are counterexamples.
\[
X_n = \frac 1 n T_n(2\cos(\theta)) + n^{n-1}
\begin{pmatrix}
 & & & 1\\
 & & & \\
 & & & \\
1 & & & 
 \end{pmatrix}
 \sim_{\lambda,\sigma,GLT} 0
\]
\[
Y_n = \frac 1 n T_n(2i\sin(\theta)) + n^{n-1}
\begin{pmatrix}
 & & & -1\\
 & & & \\
 & & & \\
1 & & & 
 \end{pmatrix}
  \sim_{\lambda,\sigma,GLT} 0
\]
\[
X_n+Y_n = 2\frac 1 n T_n(e^{i\theta}) + 2n^{n-1}
\begin{pmatrix}
 & & & \\
 & & & \\
 & & & \\
1 & & & 
 \end{pmatrix}
  \sim_{\sigma,GLT} 0 \quad \sim_\lambda 2e^{2\pi ix}
\]
The issue is that $X_n+Y_n$ is not Hermitian anymore, so we cannot use Fischer and Cauchy results on the distribution of eigenvalues. \\
 
The hypothesis of hermitianity is essential, since there exist counterexamples with $X_n$ diagonalizable, bounded and $\|Y_n\| = \|Y_n\|_1 = o(1)$. In fact,   $X_n=J_n+ \left(\frac 1n\right)^n e_ne_1^T$ where $J_n$ are nilpotent Jordan blocks, and $Y_n = \frac 1n e_ne_1^T$ lead to
 \[
 \serie X\sim_\lambda 0 \qquad \|Y_n\| = \|Y_n\|_1 = \frac 1n = o(1) \qquad \serie X+ \serie Y \sim_\lambda e^{i\theta}
 \]
  
There exist counterexamples even with normal and bounded sequences $\serie X$ and $\|Y_n\| = \|Y_n\|_1 = 1$. In fact $X_n=J_n+ e_ne_1^T$ and $Y_n = - e_ne_1^T$ lead to
  \[
  \serie X\sim_\lambda e^{i\theta} \qquad \|Y_n\| = \|Y_n\|_1 = 1 \qquad \serie X+ \serie Y \sim_\lambda 0
  \]

 Notice that the problem does not depend on the band of the perturbation matrix and the original matrix, since we can find permutation matrices $P_n$ such that $P_nJ_nP_n^T$ is tridiagonal and $P_ne_ne_1^TP_n^T$ bidiagonal, so that 
 \[
 P_nJ_nP_n^T \sim_\lambda 0\qquad \|P_ne_ne_1^TP_n^T\|_1=o(n) \qquad P_nJ_nP_n^T + P_ne_ne_1^TP_n^T \sim_\lambda e^{i\theta}
 \]
\subsection{Optimal Matching Distance}

Let us consider a distance on $\f C^n$ already introduced in \cite{Bhatia}, called \textbf{optimal matching distance}.
\[
d(v,w) = \min_{\sigma\in S_n}\max_{i=1,\dots,n} |v_i-w_{\sigma(i)}|
\]
This function induces a pseudometric on $\f C^n$.
From now on, we write $d(A,B)$ for the distance between the eigenvalues of $A,B$, since it induces a pseudodistance on $\f C^{n\times n}$.
Let us study how it behaves on perturbation of matrices, through the Bauer-Fike theorem (Theorem VIII.3.1 in \cite{Bhatia}).

\begin{lemma}\label{BF}
	Let $A$ be matrix diagonalizable through $A=VDV^{-1}$, where $D$ is a diagonal matrix and its eigenvalues are $\lambda_i$. Moreover, let $\delta =\frac 12 \min_{\lambda_i\ne \lambda_j}|\lambda_i-\lambda_j|$, where the minimum over an empty set is $+\infty$. If $\|N\|< \frac{\delta}{k_2(V)}$, then
	\[
	d(A,A+N)\le k_2(V)\|N\|.
	\]
\end{lemma}
\begin{proof}
	If $A$ has only one eigenvalue $\lambda$ with multiplicity $n$, then $A=\lambda I$, $\delta = +\infty$, $V=I$, and for every $N$,
	\[
	d(A,A+N) = \rho(N) \le \|N\| = k_2(V) \|N\|.
	\]
	From now on, we suppose that $A$ has at least two different eigenvalues. 
	Bauer-Fike Theorem let us find for every eigenvalue $\mu$ of $A+N$ an eigenvalue $\lambda$ of $A$ such that
	\[
	|\lambda-\mu|\le k_2(V) \|N\|.
	\]
	Consider now the segment $A + tN$ where $t$ varies in $t\in [0,1]$.  Using Corollary VI.1.6 of \cite{Bhatia}, there exist $n$ continuous functions $\lambda_i(t)$ representing the eigenvalues of $A + tN$ for every $t\in [0,1]$. Suppose now that for some $t\in [0,1]$ and for some index $i$ we have
	\[
	|\lambda_i(t) -\lambda_i(0)|>k_2(V)\|N\|.
	\]
	We can thus denote the first time when it happens with
	\[
	s = \min_i \inf \set{t | \,\,|\lambda_i(t) -\lambda_i(0)|>k_2(V)\|N\|}.
	\]
	Notice that $s\ne 1$.
	Suppose $j$ is an index such that $|\lambda_j(t) -\lambda_j(0)|>k_2(V)\|N\|$ for every $t$ in a right neighbourhood of $s$. Using the continuity of $\lambda_j$, we can infer that $|\lambda_j(s) -\lambda_j(0)|=k_2(V)\|N\|\le \delta$. Using Bauer-Fike, we know that there exists an eigenvalue $\lambda_i(0)$ of $A$ such that
	\[
	|\lambda_j(s) - \lambda_i(0)| \le sk_2(V) \|N\| \le s\delta <\delta
	\]
	so
	\[
	|\lambda_i(0) - \lambda_j(0)| \le |\lambda_j(s) - \lambda_j(0)| + |\lambda_j(s) - \lambda_i(0)| < 2\delta = \min_{\lambda_i\ne \lambda_j}|\lambda_i(0)-\lambda_j(0)|
	\]
	resulting in $\lambda_i(0) = \lambda_j(0)$, but
	\[
	|\lambda_j(s) -\lambda_j(0)|=k_2(V)\|N\| = |\lambda_j(s) -\lambda_i(0)|\le sk_2(V) \|N\|<k_2(V) \|N\|
	\]
	that is an absurd. We have thus proved that for every $t\in [0,1]$ and every $i$ 
	\[
	|\lambda_i(t) -\lambda_i(0)|\le k_2(V)\|N\|
	\]
	and in particular if $t=1$, 
	\[
	|\lambda_i(1) -\lambda_i(0)| = |\lambda_i(A+N) -\lambda_i(A)|\le k_2(V)\|N\|.
	\] 
\end{proof}

The result is sharp: $A=I$, $N=cI$. What happens when $A$ is not diagonalizable? 	In \cite{bf2}, we find a generalization of Bauer Fike on all the matrices:
\begin{lemma}\cite{bf2}\label{bf-gen}
	Let $A$ be any matrix with Jordan form $J$ and $A=XJX^{-1}$ and let $m$ be the biggest size of Jordan block inside $J$. If $k_2(X)\|N\|\le 2^{1-m}$, then for every eigenvalue $\lambda$ of $A+N$ there exists an eigenvalue $\mu$ of $A$ such that
\[
|\lambda-\mu|^m\le 2^{m-1} k_2(X)\|N\|.
\]
\end{lemma}

\begin{lemma}\label{BF2}
		Let $A$ be any matrix with Jordan form $J$ and $A=VJV^{-1}$. Let $m$ be the biggest size of Jordan block inside $J$, and $\delta =\frac 12 \min_{\lambda_i\ne \lambda_j}|\lambda_i-\lambda_j|$, where the minimum over an empty set is $+\infty$.
		If $\|N\|<  \frac{\delta^m}{2^{m-1}k_2(V)}$, then
	\[
	d(A,A+N)\le (2^{m-1} k_2(V)\|N\| )^{\frac 1m}
	\]
\end{lemma}
\begin{proof}
The proof is based on Lemma \ref{bf-gen}, and it is totally analogous to the proof of Lemma \ref{BF}.
\end{proof}

The result is sharp: $A=J$, $N=e_ne_1^T$. An easy corollary is the following.

\begin{corollary}\label{pert}
	If $A$ is any matrix and $\ve>0$, then there exists $\delta>0$ such that       	\[                   \|N\| \le \delta \implies d(A,A+N)\le \ve.
	\]
\end{corollary}
\begin{proof}
	If $A=VJV^{-1}$ is the Jordan form, $m$ is the biggest size of the Jordan blocks and $\gamma =\frac 12 \min_{\lambda_i\ne \lambda_j}|\lambda_i-\lambda_j|$, then we can use
	Lemma \ref{BF2}, and consider
	\[
	\|N\| \le \frac{\min\{\gamma^m,\ve^m\}}{2^{m-1}k_2(V)}=\delta
	\]
	to conclude.
\end{proof}
In general, these results entice that
\[
d(A,A+N) = O(\|N\|^{1/n}).
\]

\noindent We can now start considering sequences of matrices and define 
\[
d(\serie A,\serie B) = \limsup_{n\to \infty}\min_{\sigma\in S_n}\max_{i=1,\dots,n} |\lambda_i(A_n)-\lambda_{\sigma(i)}(B_n)|
\]
that respects the axioms of pseudometric, but it may take the value $+\infty$. A more general distance on sequences is the \textit{generalized optimal matching distance}, introduced in \cite{barb}.
\[
d'(A,B) = \min_{\sigma\in S_n}\min_{i=1,\dots,n}\left\{ \frac{i-1}n + |\lambda(A)-\lambda_{\sigma}(B)|_i^{\downarrow}  \right\}\]\[
d'(\serie A,\serie B) = \limsup_{n\to \infty}d'(A_n,B_n)
\]
where $|\lambda(A)-\lambda_{\sigma}(B)|^{\downarrow}$ is the vector of $ |\lambda_i(A)-\lambda_{\sigma(i)}(B)|$ sorted in decreasing order
$$|\lambda(A)-\lambda_{\sigma}(B)|_i^{\downarrow} \ge |\lambda(A)-\lambda_{\sigma}(B)|_j^{\downarrow} \iff i\le j.  $$
It has been proved that $d'$ induces a complete pseudometric in the space of sequences, and a lot of connections with spectral measures. Here we report one of the most important theorems of the previous paper.

\begin{theorem}\cite{barb}\label{d'}
	If $\serie A\sim_\lambda f(x)$, then
	\[
	\serie B\sim_\lambda f(x) \iff d'(\serie A,\serie B)=0.
	\]
\end{theorem}

It is easy to check that $d'(\serie A,\serie B)\le d(\serie A,\serie B)$ leading to an easy corollary.

\begin{corollary}\label{d}
	If $\serie A\sim_\lambda f(x)$, then
	\[
	\serie B\sim_\lambda f(x) \impliedby d(\serie A,\serie B)=0.
	\]
\end{corollary}

\begin{theorem}
	Given $\serie A\sim_\lambda f$, there exists a sequence of $\ve_n>0$ such that
	\[
	\|N_n\|\le \ve_n \implies \{A_n + N_n\}_n\sim_\lambda f
	\] 
\end{theorem}
\begin{proof}
	Using Corollary \ref{pert}, we can find  $\ve_n$ such that $d(A_n, A_n+N_n) \le 1/n$, so that $d(\serie A,\{A_n + N_n\}_n)=0$ and  Corollary \ref{d} leads to the thesis.
\end{proof}

Eventually, we can explore the connections between $d'$ and $d_{acs}$.

\begin{lemma}\cite{barb}\label{permm}
	Given $\serie D,\serie{D'} \in \mathscr  E$ sequences of diagonal matrices, there exists a sequence $\serie P$ of permutation matrices such that
	\[
	d'(\serie {D'},\serie{D}) =   d_{acs}(\serie{D'},\{P_nD_nP_n^T\}_n).
	\]
\end{lemma}

\noindent Actually, the proof of the last result given in \cite{barb} prove that the permutations $P_n$ found realize the minimum in
\[
d'(\serie {D'},\serie{D}) =  \min_{\serie P} d_{acs}(\serie{D'},\{P_nD_nP_n^T\}_n).
\]
If we denote as $\serie {D(A)}$ and $\serie {D(B)}$ the diagonal sequences composed by the eigenvalues of $\serie A$ and $\serie B$, then  
\[
d'(\serie A,\serie B) = d'(\serie {D(A)}, \serie {D(B)}) = \min_{P_n} \dacs{\serie {D(A)}}{\{ P_nD(B)_nP_n^T \}_n}
\]
A nice reversal relation between the two pseudodistance is the following.

\begin{lemma}
	\[
	d'(\serie A,\serie B) \ge \inf_{\serie M,\serie N} d_{acs}(\{M^{-1}_nA_nM_n\}_n, \{N_n^{-1}B_nN_n\}_n)
	\]
where the inf is taken among all sequences of invertible matrices.
\end{lemma}
\begin{proof}
It is sufficient to prove that for every couple of matrices $A,B$ with the same dimension, the following inequality holds.
	\[
	d'(A,B) \ge \inf_{M,N} d_{acs}(M^{-1}AM,N^{-1}BN)
	\]
From Lemma \ref{permm} and successive speculations, we find permutation matrices $P_n$ such that
	\[
	d'(A,B) = d_{acs}(D(A),PD(B)P^T)
	\]
Let $M,N$ be invertible matrices that bring $A,B$ to a bidiagonal upper triangular form with all the elements on the upper diagonal of norm less than $\ve>0$, and the eigenvalues ordered as in $D(A)$ and $PD(B)P^T$.
\begin{align*}
	d'(A,B) &= d_{acs}(D(A),PD(B)P^T)\\
	&\ge 
	-d_{acs}(D(A),M^{-1}AM)
	+d_{acs}(M^{-1}AM,N^{-1}BN)
	-d_{acs}(N^{-1}BN,PD(B)P^T)\\
	&\ge
	d_{acs}(M^{-1}AM,N^{-1}BN)
	-2\ve\\
	\implies 
	d'(A,B) &\ge \inf_{M,N} d_{acs}(M^{-1}AM,N^{-1}BN).
\end{align*}
\end{proof}

\noindent The result is not sharp. In fact, if
$A_n=J_n$ and $B_n=J_n + e_ne_1^T$, we know that 
\[
d'(\serie A,\serie B) = 1,\qquad \dacs{\serie A}{\serie B}=0.
\]

\subsection{Other Distances}

An idea to solve the perturbation problem is to find pseudodistances on matrices and sequences so that any couple of sequences at zero distance admit the same spectral symbol.

\[
d_\lambda (\serie A,\serie B) = 0 \qquad \serie A\sim_\lambda k \implies \serie B\sim_\lambda k
\]

Define the function
\[
d_N(\serie A,\serie B) = \limsup_{n\to \infty}\frac 1 n \|A_n-B_n\|_1 = p_N(\serie A - \serie B).
\]
$d_N$ respects the properties of a distance, but it may take infinite value. On bounded sequences, the function is a complete pseudodistance thanks to Lemma \ref{compl}. An other complete pseudodistance that employs only the rank of the matrices is
\[
d_R(\serie A,\serie B) = \limsup_{n\to\infty } \frac{\rk(A_n-B_n)}{n} = p_R(\serie A - \serie B).
\]

Notice that acs distance behaves well only on hermitian sequences, so the aim is to penalize the non-hermitianity of matrices.
\[
d_H(\serie A,\serie B) = d_{acs}(\serie {\mc R(A)},\serie {\mc R(B)}) + p_{N}(\serie {\mc I(A)}) + p_{N}(\serie {\mc I(B)})
\]
$d_H$ is a complete pseudodistance that may take infinite value, even if we replace $ d_{acs}(\serie {\mc R(A)},\serie {\mc R(B)})$ with $ d_{acs}(\serie A,\serie B)$.

We will use $d_H$ in the next section and show it is equivalent to Conjecture \ref{conj} and \ref{inv}.

\section{Equivalent Statements}

Here we report some results discovered while working on the main problem.

\begin{lemma}
The statement 
\[
d_H(\serie A,\serie B)=0, \,\,\serie A\sim_\lambda f\implies \serie B \sim_\lambda f
\]
holds if and only if both the following are true.
\begin{enumerate}
\item If $\serie X\sim_\lambda k$ is a sequence of Hermitian matrices and $\|Y_n\|_1=o(n)$, then $\serie X +\serie Y\sim_\lambda k$
\item If $\serie A\sim_\lambda k$ and $\|\Im(A_n)\|_1=o(n)$, then $\{ \Re(A_n)  \}_n\sim_\lambda k$
\end{enumerate}
\end{lemma}
\begin{proof}
If the statement on $d_H$ is true, then both 1. and 2. are true, since 
\[
\|\Re(Y_n) \|_1 \le \|Y_n\|_1=o(n),\qquad \|\Im(Y_n)\|_1 \le \|Y_n\|_1=o(n) \implies  
\]
\[
d_H(\serie X, \serie X+\serie Y) = p_{acs}(\{ \Re(Y_n)  \}_n) + p_N(\serie 0) + p_N(\{ \Im(Y_n)  \}_n) =0 
\]
and
\[
d_H(\serie A, \{ \Re(A_n)  \}_n) = p_{acs}(\serie 0) + p_N(\{ \Im(A_n)  \}_n) + p_N(\serie 0) =0.
\]\\

\noindent If 1. and 2. are true, consider $\serie A$ e $\serie B$ such that $d_H(\serie A,\serie B)=0$ and  $\serie A\sim_\lambda f$. Then
\[
d_{acs}(\{ \Re(A_n)  \}_n,\{ \Re(B_n)  \}_n )= p_N(\{ \Im(A_n)  \}_n) = p_N(\{ \Im(B_n)  \}_n) =0,
\]
so we can apply 2. and obtain $\{ \Re(A_n)  \}_n\sim_\lambda f$. The sequences $\{ \Re(A_n)  \}_n,\{ \Re(B_n)  \}_n$ are Hermitian and identified by $d_{acs}$, so 
$\{ \Re(B_n)  \}_n\sim_\lambda f$. Eventually, $p_N(\{ \Im(B_n)  \}_n) =0$ implies that  $\|\Im(B_n)\|_1 =o(n)$, so we use 1. and conclude that
\[
\serie B = \{ \Re(B_n)  \}_n + \{ \Im(B_n)  \}_n\sim_\lambda f.
\]
\end{proof}

\begin{lemma}
Consider the following statements.
\begin{enumerate}
\item If $\serie X\sim_\lambda k$ is an Hermitian sequence and $\|Y_n\|_1=o(n)$, then $\serie X +\serie Y\sim_\lambda k$
\item If $\serie X\GLT k$ is an Hermitian sequence and $\|Y_n\|_1=o(n)$, then $\serie X +\serie Y\sim_\lambda k$
\item If $\serie X\GLT k$ and $\|\Im(X_n)\|_1=o(n)$, then $\serie X\sim_\lambda k$
\item If $\serie D\GLT k$ is a sequence of real diagonal matrices and $\serie Y$ is a sequence of skew-Hermitian matrices with $\|Y_n\|_1=o(n)$, then $\serie D + \serie Y\sim_\lambda k$
\item If $\serie X\sim_\lambda k$ is an Hermitian sequence and $\serie D$ is a sequence of real diagonal matrices with $\|D_n\|_1=o(n)$, then $\serie X + i\serie D\sim_\lambda k$
\item If $\serie X\GLT k$ is an Hermitian sequence and $\serie D$ is a sequence of real diagonal matrices with $\|D_n\|_1=o(n)$, then $\serie X + i\serie D\sim_\lambda k$
\end{enumerate}
Statements 1. 2. 3. 4. 5. are equivalent and they all implies 6.
\end{lemma}
\begin{proof}
Let us show the chain of implications 
\[
 1. \implies 2. \implies 3. \implies 4. \implies 1.
\]
and
\[
1.\implies 5. \implies 4.\,\, 6.
\]

\noindent $1.\implies2.)$  Remembering Lemma \ref{GLT:hermitian_spectral}, we know that if  $X_n$ are Hermitian, then          
\[
\serie X\GLT k\implies \serie X\sim_\lambda k.\]

\noindent $2.\implies3.)$ 
\[
\serie X\GLT k \implies \{\Re(X_n)\}\GLT \Re(k),\qquad \{\Im(X_n)\}\GLT \Im(k)
\]
but $\|\Im(X_n)\|_1=o(n)$ is a zero-distributed sequence thanks to Lemma \ref{zero:norm_1}, so $\Im(k)=0$ and consequentially $\{\Re(X_n)\}\GLT \Re(k)=k$. Using 2., we obtain
	\[
	\serie X = \{\Re(X_n) + i\Im(X_n)\}\sim_\lambda k.
	\]
\noindent$3.\implies 4.)$
 $\|Y_n\|_1=o(n)$ is a zero-distributed sequence thanks to Lemma \ref{zero:norm_1}, so we use the algebra structure of the GLT space to obtain $\{D_n+Y_n\}\GLT k$. Notice that
 \[
 \Re(D_n+Y_n) = D_n,\quad \Im(D_n+Y_n) = Y_n,
 \]
so we apply 3. on $X_n=D_n+Y_n$ and conclude 
 $\serie D + \serie Y\sim_\lambda k$\\

\noindent$4.\implies 1.)$ 
Let $X_n = Q_nD_nQ_n^H$ be a diagonalization of the Hermitian matrices $X_n$.  If $\serie X\sim_\lambda k$, then $\serie D\sim_\lambda k$, since they have the same eigenvalues. Given $h:[0,1]\to \f C$ a rearranged version of $k$, $\serie D\sim_\lambda h$ holds and we can use Theorem \ref{LAMBDAtoGLT} to find permutation matrices $P_n$ such that $\{P_nD_nP_n^T\}_n\GLT h\otimes 1$. Notice that $\|\Re(Y_n)\|_1 \le (\|Y_n\|_1 + \|Y_n^H\|_1) /2 = \|Y_n\|_1$, so $\|Y_n\|_1 = o(n)\implies \|\Re(Y_n)\|_1 =o(n)$ and the same holds for $\|\Im(Y_n)\|$. In particular they are both zero-distributed sequences  thanks to Lemma \ref{zero:norm_1}. Notice that  $\|P_nQ_n^H \Re(Y_n) Q_nP^T_n\|_1 = o(n)$, so the sequence is also zero-distributed, and
\[
\{P_nD_nP_n^T + P_nQ_n^*\Re(Y_n)Q_nP_n^T \}\GLT h\implies \serie X + \{\Re(Y_n)\}\sim_\lambda k
\] 
Repeat the reasoning with the Hermitian matrices $X_n + \Re(Y_n) = U_nD'_nU_n^*$, where  $\serie{D'}\GLT h\otimes 1$, and since $\|U_n^*\Im(Y_n)U_n\|_1=o(n)$, we can use 4. to conclude
\[
\serie{D'} + \{U_n^*\Im(Y_n)U_n\} \sim_\lambda h  \implies\serie X +\serie Y\sim_\lambda k
\]

\noindent$1.\implies 5.)$ Just notice that $\|\text iD_n\|_1 = \|D_n\|_1 = o(n)$.

\noindent$5.\implies 4.)$ 
Given $\serie Y$ skew-Hermitian matrices with
$\|Y_n\|_1=o(n)$, let $D'_n$ be diagonal real matrices such that  $Y_n=iQ_nD'_nQ_n^H$ is an unitary diagonalization. If $\serie D\GLT k$ are diagonal real matrices, then they are Hermitian and thanks to Lemma \ref{GLT:hermitian_spectral}, $\serie D\sim_\lambda k$ and $\{ Q_n^HD_nQ_n \}_n\sim_\lambda k$ . Using 5. we conclude
\[
\{ Q_n^HD_nQ_n + iD'_n \}_n\sim_\lambda k \implies \{ D_n + Y_n \}_n\sim_\lambda k 
\]
\noindent$5.\implies 6.)$ Use Lemma \ref{GLT:hermitian_spectral}.

%
\end{proof}

\section{Perturbation of Normal sequences}

When dealing with normal matrices instead of Hermitian matrices, we get different results.

\begin{lemma}
Let $X_n$ be normal matrices, with $\serie X\sim_\lambda f$.Consider the following statements:
\begin{enumerate}
\item $\serie Y$ zero-distributed and $X_n+Y_n$ normal,
\item $\|Y_n\|_p = o(1)$ where $\|\cdot\|_p$ is the $p$ Schatten norm for some $1\le p\le 2$,
\item $\|Y_n\|_p = o(n^{\frac 2p-1})$ where $\|\cdot\|_p$ is the $p$ Schatten norm for some $2\le p <\infty$,
\item $\|Y_n\|=o(\frac 1n)$.
\end{enumerate}
If any of them holds, then
\[
\serie X+\serie Y \sim_\lambda f.
\]
\end{lemma}
\begin{proof}Remember that Schatten norms respects $\|A\|_p\le \|A\|_q$ for $\infty \ge p\ge q\ge 1$.\\

\noindent$1.)$ A normal sequence $\serie X\sim_\lambda f$ respects the hypothesis of Lemma 11 in \cite{norm}, and $\serie Y$ is zero-distributed so
\[
\{ X_n+Y_n - cI_n \}_n\ea \{ X_n - cI_n \}_n \sim_\sigma f(z) - c \qquad\forall c\in \f C.
\]
Given that $X_n+Y_n$ is normal, using Lemma 12 of \cite{norm}, we conclude that $\{X_n+Y_n\}_n\sim_\lambda k$.\\

\noindent$2.)$ It is sufficient to prove it when $p=2$. Using Problem VI.8.11 in \cite{Bhatia}, we know that for  every normal matrix $A$ and any other matrix $B$, we have 
\[
\min_{\sigma\in S_n}\left( \sum_{i=1}^n |\lambda_i-\mu_{\sigma(i)}|^2 \right)^{1/2} \le \sqrt n\|A-B\|_2.
\]
If  $\tau$ is the permutation that realizes the minimum for $X_n$ and $X_n+Y_n$, and $k_n$ is the number of indices such that
$ |\lambda_i-\mu_{\sigma(i)}|>\ve $, then
\[
 \sqrt {k_n}\ve \le \left( \sum_{i=1}^n |\lambda_i-\mu_{\sigma(i)}|^2 \right)^{1/2} \le \sqrt n\|Y_n\|_2,\qquad 
\frac {k_n}n \le  \left(\frac{\|Y_n\|_2}{\ve}\right)^2\xrightarrow{n\to\infty} 0.
\]
Consequentially,
\[
d'(\serie X,\serie X+\serie Y) \le  \limsup_{n\to\infty} \frac{k_n}{n} + \ve = \ve
\]
for every $\ve>0$, so we use Theorem \ref{d'} and conclude
\[
\serie X+\serie Y \sim_\lambda f.
\]
\noindent$3.)$ Using Problem VI.8.11 in \cite{Bhatia}, , we know that for  every normal matrix $A$ and any other matrix $B$, we have 
\[
\min_{\sigma\in S_n}\left( \sum_{i=1}^n |\lambda_i-\mu_{\sigma(i)}|^p \right)^{1/p} \le n^{1-1/p}\|A-B\|_p \quad \forall\,p\ge 2.
\]
If  $\tau$ is the permutation that realizes the minimum for $X_n$ and $X_n+Y_n$, and $k_n$ is the number of indices such that
$ |\lambda_i-\mu_{\sigma(i)}|>\ve $, then
\[
 k_n^{1/p}\ve \le \left( \sum_{i=1}^n |\lambda_i-\mu_{\sigma(i)}|^p \right)^{1/p} \le n^{1-1/p}\|Y_n\|_p,\qquad 
\frac {k_n}n \le  \left(\frac{\|Y_n\|_p}{\ve n^{\frac 2p -1 }}\right)^p\xrightarrow{n\to\infty} 0
\]
consequentially
\[
d'(\serie X,\serie X+\serie Y) \le  \limsup_{n\to\infty} \frac{k_n}{n} + \ve = \ve
\]
for every $\ve>0$, so  we use Theorem \ref{d'} and conclude
\[
\serie X+\serie Y \sim_\lambda f.
\]
\noindent $4.)$ Same proof as 3., with $p=\infty$.
\end{proof}


\begin{thebibliography}{99}
	\footnotesize
	
	\bibitem{Barbarino2017}
	{\sc Barbarino G.}
	{\em Equivalence between GLT sequences and measurable functions.}
	Linear Algebra Appl. 529 (2017) 397--412.
	
	
	\bibitem{TR} 
	{\sc Barbarino G.}
	{\em Diagonal Matrix Sequences and their Spectral Symbols.}
	\url{http://arxiv.org/abs/1710.00810} (2017)
	
	
	\bibitem{barb}
	{\sc Barbarino G.}
	{\em Spectral Measures.} 
	Proceedings of Cortona Meeting, Springer INdAM Series (to appear 2018).
	
	
	\bibitem{norm}
	{\sc Barbarino G.}
	{\em Normal Form for GLT Sequences.}
	\url{https://arxiv.org/abs/1805.08708} (2018) 
	
	
	\bibitem{BG-ELA2017}
	{\sc Barbarino G., Garoni C.}
	{\em From convergence in measure to convergence of matrix-sequences through concave functions and
		singular values.}
	Electr. J. Linear Algebra  32 (2017) 500--513. 
	
	\bibitem{Per1}
	{\sc Barbarino G., Serra-Capizzano S.}
	{\em Non-Hermitian perturbations of Hermitian matrix-sequences and applications to the spectral analysis of approximated PDEs.}
	Technical report / Department of Information Technology, Uppsala University (2018).
	
	\bibitem{Bhatia}
	{\sc Bhatia R.}
	{\em  Matrix Analysis.}
	Springer, New York (1997).
	
	
	
	
	\bibitem{joint-Albrecht}
	{\sc B\"ottcher A., Garoni C., Serra-Capizzano S.}
	{\em Exploration of Toeplitz-like matrices with unbounded symbols is not a purely academic journey.}
	Sb. Math. 208 (2017) 1602--1627.
	
	
	
	\bibitem{BS}
	{\sc B\"ottcher A., Silbermann B.}
	{\em Introduction to Large Truncated Toeplitz Matrices}.
	Springer, New York (1999).
	
	
	\bibitem{radial}
	{\sc Cao F. L., Xie T. F.}
	{\em The rate of approximation of Gaussian radial basis neural networks in continuous function space.}
	Acta Mathematica Sinica, English Series 29, (2013), 295--302.
	
	
	\bibitem{Garoni}
	{\sc  Garoni C.}
	{\em  Topological foundations of an asymptotic approximation theory for sequences of matrices
		with increasing size.}
	Linear Algebra Appl. 513 (2017) 324--341.
	
	
	\bibitem{Grudsky-paper}
	{\sc Garoni C., Serra-Capizzano S.}
	{\em The theory of locally Toeplitz sequences: a review, an extension, and a few representative applications.}
	Bol. Soc. Mat. Mex. 22 (2016) 529--565.
	
	\bibitem{GLT-book}
	{\sc Garoni C., Serra-Capizzano S.}
	{\em Generalized Locally Toeplitz Sequences: Theory and Applications (Volume I).}
	Springer, Cham (2017).
	
	
	

	

	
	
	\bibitem{GLT-bookII} 
	{\sc Garoni C., Serra-Capizzano S.} 
	{\em Generalized Locally Toeplitz Sequences: Theory and Applications.} Technical Report 2017-002, Uppsala University (2017). Preliminary version of: {\sc Garoni C., Serra-Capizzano S.} {\em Generalized Locally Toeplitz Sequences: Theory and Applications (Volume II).} In preparation for Springer.
	
	
	
	
	
	\bibitem{Albrecht-paper}
	{\sc Garoni C., Serra-Capizzano S.}
	{\em The theory of generalized locally Toeplitz sequences: a review, an extension, and a few representative applications.}
	Oper. Theory Adv. Appl. 259 (2017) 353--394.
	
	
	
	\bibitem{vassalos}
	{\sc Garoni C., Serra-Capizzano S., Vassalos P.}
	{\em A general tool for determining the
		asymptotic spectral distribution of Hermitian matrix-sequences}
	Oper. Matrices 9 (2015) 549--561.
	
		\bibitem{golinskii}
		{\sc Golinskii L.,  Serra-Capizzano S.}
		{\em The asymptotic properties of the spectrum of nonsymmetrically perturbed Jacobi matrix sequences.} J. Approx. Theory 144 (2007) 84--102.
		
	
	\bibitem{szego}
	{\sc Grenader U., Szeg\"o G.}
	{\em Toeplitz Forms and Their Applications.}
	Second Edition, AMS Chelsea Publishing, New York (1984).
	
	
	\bibitem{glt1}
	{\sc Serra-Capizzano S.}
	{\em Generalized locally Toeplitz sequences: spectral analysis and applications to discretized partial differential equations.}
	Linear Algebra Appl. 366 (2003) 371--402.
	
	\bibitem{glt2}
	{\sc Serra-Capizzano S.}
	{\em The GLT class as a generalized Fourier analysis and applications.}
	Linear Algebra Appl. 419 (2006) 180--233.
			 
			\bibitem{bf2}
			{\sc Shi X., Wei Y.}
			{\em A sharp version of Bauer–Fike’s theorem.}
			Journal of Computational and Applied Mathematics 236, n. 13 (2012), 3218--3227. 
			 
			 \bibitem{Tilliloc}
			 {\sc Tilli P.}
			 {\em Locally Toeplitz sequences: spectral properties and applications.}
			 Linear Algebra Appl. 278 (1998) 91--120.
			 
\end{thebibliography}
\end{document}